\documentclass{article}[11pt,a4paper]

\usepackage[english]{babel}

\usepackage{amsmath}
\usepackage{amsfonts}
\usepackage{amsthm}
\usepackage{amssymb}
\usepackage{bm}

\usepackage[latin1]{inputenc}

\usepackage{enumerate}
\usepackage{multicol}

\usepackage[all,cmtip]{xy}

\newtheoremstyle{mystyle2}{}{}{}{2pt}{\scshape}{.}{ }{}
\newtheoremstyle{mystyle}{}{}{\slshape}{2pt}{\scshape}{.}{ }{}
\newtheoremstyle{etapestyle}{}{}{\itshape}{2em}{\sffamily}{:}{ }{\thmname{#1}}
\newtheoremstyle{definitionstyle}{}{}{}{2pt}{\bfseries}{.}{ }{}

\newtheorem{thm}{Theorem}
\newtheorem{cor}[thm]{Corollary}

\newtheorem{prop}[thm]{Proposition}

\newtheorem{lemme}[thm]{Lemma}

\newtheorem*{prop*}{Proposition}

\theoremstyle{mystyle2}

\theoremstyle{mystyle}

\theoremstyle{remark}

\theoremstyle{etapestyle}

\theoremstyle{definitionstyle}

\DeclareMathOperator{\alg}{alg}

\DeclareMathOperator{\Pic}{Pic}

\DeclareMathOperator{\Id}{Id}

\DeclareMathOperator{\Cl}{Cl}
\DeclareMathOperator{\NS}{NS}
\DeclareMathOperator{\Alb}{Alb}

\DeclareMathOperator{\sm}{sm}

\begin{document}
\renewcommand\refname{References}
\title{Quasi-projectivity of normal varieties}
\author{Olivier BENOIST \\
  \multicolumn{1}{p{.9\textwidth}}{\centering\normalsize{\emph{ENS, DMA, 45 rue d'Ulm}}}\\
  \multicolumn{1}{p{.9\textwidth}}{\centering\normalsize{\emph{75005 Paris, FRANCE}}}}
\date{}
\maketitle
\begin{abstract} 
We prove that a normal variety contains finitely many maximal quasi-projective open subvarieties.
As a corollary, we obtain the following generalization of the Chevalley-Kleiman projectivity criterion :
a normal variety is quasi-projective if and only if every finite subset is contained in an affine open subvariety.
The proof builds on a strategy of W\l odarczyk, using results of Boissi\`ere, Gabber and Serman. 
 \end{abstract}




Let $k$ be an algebraically closed field. A variety is a separated and integral $k$-scheme of finite type.
The goal of this text is to prove the following:

\begin{thm}\label{princ}
A normal variety $X$ contains finitely many maximal quasi-pro\-jec\-tive open subvarieties.
\end{thm}

The hypothesis that $k$ is algebraically closed is mainly for convenience. We will explain later how to remove it (see Theorem \ref{princK}).

Theorem \ref{princ} was proved by W\l odarczyk 
in \cite{Wlo} 
when $X$ has a normal compac\-tification $\bar{X}$ such that
$\dim_{\mathbb{Q}}(\Cl(\bar{X})/\Pic(\bar{X}))\otimes\mathbb{Q}<\infty$.
In particular, this settled the case where $X$ is complete and $\mathbb{Q}$-factorial,
or smooth of characteristic $0$. 

As W\l odarczyk points out, these statements imply generalizations of the Chevalley-Kleiman projectivity criterion
(proved by Kleiman \cite{Kleimnum} for complete and $\mathbb{Q}$-factorial normal varieties, see also \cite{Hartshorneample} Theorem I 9.1).
Accordingly, we obtain:

\begin{cor}\label{Kleiman}
 A normal variety $X$ is quasi-projective if and only if every finite subset of $X$ is contained in an affine open subvariety of $X$.
\end{cor}

As explained in \cite{Wlo} Theorem D, it is possible to apply
this quasi-projectivity criterion to prove the following corollary. It is very close to results of Raynaud (\cite{Raynaudample} V Corollaire 3.14) and Brion
(\cite{Brionaction} Theorem 1), and could also have been proved by their methods.

\begin{cor}\label{groupaction}
 Let $G$ be a connected algebraic group acting on a normal variety $X$, and let $U$ be a quasi-projective open subvariety of $X$. Then $G\cdot U$ is
quasi-projective.
\end{cor}

That Theorem \ref{princ} holds for any normal variety was conjectured by Bia\l ynicki-Birula in \cite{BB}.
The reason for Bia\l ynicki-Birula's interest in this conjecture was that, using the results
of \cite{BB}, it implies the following corollary:

\begin{cor}
 In characteristic $0$, a normal variety on which a reductive group $G$ acts contains finitely many open subvarieties that admit a
quasi-projective good quotient and that
are maximal with respect to $G$-saturated inclusion.
\end{cor}

  Note that, by examples of Serre described in \cite{Ferrand}, 6.2, 6.3
(see also \cite{Jelonek}), those statements are not true in general for nonnormal varieties.

\vspace{1em}

  Let us explain why those statements are more difficult in the normal case than in the smooth case.
If $X$ is a complete normal surface
such that every finite subset of $X$ is contained in an affine open subvariety of $X$,
it should be projective by Corollary \ref{Kleiman}. In particular, we should have $\Pic(X)\neq 0$.
When $X$ is smooth, this is not a surprise because any nonzero effective Weil divisor is Cartier
and has nontrivial class in the Picard group. However, there exist complete normal surfaces with trivial
Picard group (see \cite{Schroer}). This shows that the
hypothesis about finite subsets of $X$ is needed even to prove that $\Pic(X)\neq 0$.

For this reason, we need techniques to construct Cartier divisors on normal varieties. This is the
role of the technical hypothesis in W\l odarczyk's theorem. For this purpose, we
will use the work of Boissi\`ere, Gabber, and Serman \cite{BGS}, and a theorem of P\'epin \cite{Pepin}.

\vspace{1em}

Let us turn to the proof of Theorem \ref{princ}.
The following proposition removes the hypothesis $\dim_{\mathbb{Q}}(\Cl(X)/\Pic(X))\otimes\mathbb{Q}<\infty$
in Lemma 4 of \cite{Wlo}.

\begin{prop} \label{proputile}
Let $X$ be a complete normal variety. Then there exists a po\-si\-tive integer $r$ such that the following holds.
Let $U_1,\dots, U_s$ be
quasi-projective open subvarieties of $X$ with $s>r$
such that for any $1\leq i,j\leq s$, $U_i\setminus (U_i\cap U_j)$ is of codimension $\geq 2$ in $U_i$.
Then there exist
distinct indices $i_1,\dots,i_m,j_1,\dots,j_n\in\{1,\ldots, s\}$ with $m,n\geq 1$
such that $U=(U_{i_1}\cap\dots\cap U_{i_m})\cup(U_{j_1}\cap\dots\cap U_{j_n})$
is quasi-projective. 
\end{prop}

Using this proposition, W\l odarczyk's proof goes through:

\begin{proof}[$\mathbf{Proof \text{ }of \text{ }Theorem\text{ }\ref{princ}}$]~
Compactifying $X$ by Nagata's theorem, and normalizing this compactification, we may suppose that $X$ is complete.
We may then follow the proof of Theorem A of \cite{Wlo} using our Proposition \ref{proputile} instead of Lemma 4 of loc. cit.
\end{proof}

We are now reduced to proving Proposition \ref{proputile}. There are two steps where W\l odarczyk's arguments
break down under our more general hypotheses. The first is to construct a suitable Weil divisor on $X$ that is Cartier on $U$. The
second is to show that this divisor
is ample on $U$ via a numerical criterion, although it may not be Cartier on all of $X$.
Both of these difficulties will be overcome using results of \cite{BGS}.

To deal with the first one, we will use
\cite{BGS} 6.7, which is stated in \cite{BGS} as a corollary of the proof of their Th\'eor\`eme 6.1.
Since it is crucial for our needs, we recall the statement below (Theorem \ref{BGSTh}) and develop its proof a bit.
This result had already been obtained by Bingener and Flenner in characteristic $0$ (\cite{BinFlen} Corollary 3.6).

The second one will be solved using Proposition \ref{amplessi} below, which replaces here Lemma 1 of \cite{Wlo}, and whose proof
uses \cite{BGS} again, via our Lemma \ref{blowup}.

\vspace{1em}

 Let us first recall some material from \cite{BGS}.

If $X$ is a normal variety, the set of linear equivalence classes of Weil divisors on $X$ is in bijection
with the set of isomorphism classes of rank $1$ reflexive sheaves on $X$. It is an abelian group: the class group $\Cl(X)$ of $X$.

 If
$X$ is a projective normal variety with a marked smooth point $a$, $\Alb(X)$ is the Albanese variety of $X$:
it is endowed with a rational
map $\alpha_X:X\dashrightarrow \Alb(X)$ that is universal among rational maps from $X$ to an abelian variety sending
$a$ to $0$. Let $P(X)$ be the dual abelian variety of $\Alb(X)$. The Poincar\'e bundle on
$\Alb(X)\times P(X)$ identifies $P(X)(k)$ with the set of rank $1$ reflexive sheaves on $X$ algebraically equivalent to $0$
or, equivalently,
with the set of linear equivalence classes of Weil divisors on $X$ that
are algebraically equivalent to $0$
(\cite{BGS} Proposition 3.2).
We obtain
an exact sequence: $$0\to P(X)(k)\to\Cl(X)\to \NS(X)\to 0,$$ where $\NS(X)$ is the abelian group of algebraic equivalence classes
of Weil divisors on $X$:
the N\'eron-Severi group of $X$. The group $\NS(X)$ is of finite type by \cite{BrunoKahn} Th\'eor\`eme 3.

We see from this construction that $P(X)$ is a birational invariant of $X$. Moreover, if $p:X^0\to X$ is a birational morphism
between projective normal varieties, and if $D\equiv_{\alg}0$ is a Weil divisor on $X^0$, then the linear equivalence classes of
$D$ and $p_*D$ are represented by the same element of $P(X^0)(k)=P(X)(k)$.

 The proof of \cite{BGS} Th\'eor\`eme 6.1 shows that, if $x\in X$, the subset of $P(X)(k)$ consisting
of classes that are Cartier at $x$ is
the set of $k$-points of a closed subgroup $A_x$ of $P(X)$.  
If $U$ is an open subset of $X$, we will denote by $P(X)_U$ the set-theoretical intersection $\bigcap_{x\in U} A_x$:
it is a closed subgroup of $P(X)$ whose $k$-points represent the classes
that are Cartier on $U$. Of course, if $p:X'\to X$ is a birational morphism that is an isomorphism over $U$, $P(X')_U=P(X)_U$.
Let us denote by $P(X)_U^0$ the connected component of $0$ in $P(X)_U$.

 The following lemma will be used in the proof of Proposition \ref{amplessi}.
We recall that a $U$-admissible blow-up is a blow-up whose center is disjoint from $U$.

\begin{lemme}\label{blowup}
 Let $X$ be a projective normal variety, and let $U$ be an open subvariety. Then there exists a
$U$-admissible blow-up $X'\to X$ such that
every class in $P(X)_U^0(k)$ is Cartier on the normalization $\tilde{X'}$.
\end{lemme}

\begin{proof}[$\mathbf{Proof}$]~
Let $X^{\sm}$ be the smooth locus of $X$, and let $j:X^{\sm}\to X$ be the inclusion.
Following the proof of \cite{BGS} Th\'eor\`eme 4.2 closely, we will construct a universal
line bundle on $(X^{\sm}\cup U)\times P(X)_U^0$.
To do so, let $\mathcal{E}_U$ be the universal line bundle on $X^{\sm}\times P(X)_U^0$
(it is the restriction of the universal line bundle
$\mathcal{E}$ on $X^{\sm}\times P(X)$), and let $\mathcal{L}_U=(j\times\Id)_*\mathcal{E}_U$.
By \cite{BGS} Lemme 2.4, $\mathcal{L}_U$ is a coherent
sheaf on $X\times P(X)_U^0$. The first step of the proof of \cite{BGS} Th\'eor\`eme 4.2 (i.e.,
applying Ramanujam-Samuel's theorem to the projection $X\times P(X)_U^0\to X$)
shows that the locus
where $\mathcal{L}_U$ is invertible is of the form
$W_{0,U}\times P(X)_U^0$. The second step of this same proof explains how to construct
an open subset $\Omega_U$ of $P(X)_U^0$
such that if $D\equiv_{\alg}0$ is a Weil divisor on $X$ whose class $[D]$ belongs to $\Omega_U(k)$,
$\mathcal{L}_U|_{X\times[D]}\simeq\mathcal{O}_X(D)$. Hence, if $x\in\Omega_U(k)$, $\mathcal{L}_U|_{X\times\{x\}}$ is invertible
on $X^{\sm}\cup U$, so that $\mathcal{L}_U$ is invertible on $(X^{\sm}\cup U)\times\Omega_U$ by \cite{BGS} Lemme 2.6.
Combining these two steps shows that $\mathcal{L}_U$ is invertible
on $(X^{\sm}\cup U)\times P(X)_U^0$.

We may then apply \cite{Pepin} Th\'eor\`eme 2.1 (we use it in a situation very similar to the one P\'epin designed it for):
there exists a $(X^{\sm}\cup U)$-admissible
blow-up $X'\to X$ and a line bundle $\mathcal{M}$
on $X'\times P(X)_U^0$ such that $\mathcal{M}$ coincides with $\mathcal{L}_U$ over $(X^{\sm}\cup U)\times P(X)_U^0$.
Let $\mathcal{M}_0=\mathcal{M}|_{X'\times\{ 0\}}$: it is a line bundle on $X'$ that is trivial on $X^{\sm}\cup U$.
We set $\mathcal{M}'=\mathcal{M}\otimes p_1^*\mathcal{M}_0^{-1}$, and
denote by $\tilde{\mathcal{M}}'$ its pull-back to the normalization $\tilde{X'}\times P(X)_U^0$.

We are ready to show that the construction works. Let $x\in P(X)_U^0(k)$.
Consider the invertible sheaf $\tilde{\mathcal{M}}'|_{\tilde{X}'\times\{x\}}$ on
$\tilde{X'}$. It is algebraically equivalent to $\tilde{\mathcal{M}}'|_{\tilde{X'}\times \{0\}}\simeq \mathcal{O}_{\tilde{X}'}$
because $P(X)_U^0$ is connected: it corresponds to a point $y\in P(X)(k)$. 
Its restriction to $X^{\sm}$ is precisely $\mathcal{L}_U|_{X^{\sm}\times \{x\}}\simeq\mathcal{E}_U|_{X^{\sm}\times \{x\}}$.
Hence $y$ is represented on $X$ by
the reflexive sheaf $j_*(\mathcal{E}_U|_{X^{\sm}\times \{x\}})$.
This shows that $y=x$: in particular, $x$ is represented by an invertible sheaf on $\tilde{X'}$,
which is what we wanted.
\end{proof}

We may now prove:

\begin{prop}\label{amplessi}
 Let $X$ be a complete normal variety, let $U$ be an open subvariety of $X$, and let $D_1$ and $D_2$ be
Weil divisors on $X$. Suppose that
$D_1\equiv_{\alg} D_2$ and that $D_1$ and $D_2$ are Cartier on $U$. Then $D_1$ is ample on $U$
if and only if $D_2$ is ample on $U$.
\end{prop}

\begin{proof}[$\mathbf{Proof}$]~
Since $P(X)_U(k)/P(X)_U^0(k)$ is a finite group, we may suppose,
up to replacing $D_1$ and $D_2$ by $nD_1$ and $nD_2$
for some positive integer $n$, that $D_1-D_2$ represents a class in $P(X)_U^0(k)$.

Let $\pi:\tilde{X'}\to X$ be a normalized blow-up as in Lemma \ref{blowup}. It is then possible to find a Cartier divisor $\Delta$ on $\tilde{X'}$ that
is algebraically equivalent to $0$ and such that $\pi_*\Delta=D_2-D_1$. 

Choose any Weil divisor $D_1'$ on $\tilde{X'}$ such that $\pi_* D_1'=D_1$, and set $D_2'=D'_1+\Delta$. By \cite{Wlo} Lemma 1,
we see that $D'_1|_U$ is ample if and only if $D'_2|_U$ is ample.
Since $D_1|_U=D_1'|_U$ and  $D_2|_U=D_2'|_U$, the result follows.
\end{proof}

To state the last result of \cite{BGS} that we will need, we will use the following notation.
If $Z$ is a subset of a normal variety $X$, we say that $Z$ has property
$(\ast)$ (resp. $(\ast_{0})$) if for any $x,y\in Z$, a Weil divisor on $X$ (resp. that is algebraically equivalent to $0$)
is Cartier at $x$ if and only if it is Cartier at $y$. 

\begin{thm}[\cite{BGS} 6.7]\label{BGSTh}
Let $X$ be a normal variety. Then there exists a finite partition $X=\bigsqcup_{k=1}^N Z_k$ by irreducible locally closed subsets
with property $(\ast)$. 
\end{thm}

\begin{proof}[$\mathbf{Proof}$]~
The existence of such a partition is local on $X$. We may thus suppose that $X$ is affine.
Taking its closure in projective space, and normalizing this compactification, we may suppose $X$ is projective.

The proof of Th\'eor\`eme 6.1 of \cite{BGS} shows the existence of a partition
$X=\bigsqcup_{l=1}^M Y_l$ by irreducible locally closed subsets
with property $(\ast_0)$.
We will refine this partition to obtain one as we want. To do this, it suffices, by noetherian induction,
to prove the following: for any $l$, there exists an open subset $Z_l$ of $Y_l$ with property $(\ast)$.

Let $\eta$ be the generic point of $Y_l$ and let $\Gamma$ be the subgroup of $\NS(X)$ consisting
of elements that have representatives that are
Cartier at $\eta$. Since $\NS(X)$ is of finite type, $\Gamma$ is of finite type,
and we may find Weil divisors $D_1,\ldots, D_t$ on $X$ that are Cartier at $\eta$ and
whose images in $\NS(X)$ generate $\Gamma$. Let $Z_l$ be the open subset of $Y_l$ where $D_1,\ldots, D_t$ are Cartier.
Let us show that $Z_l$ has property $(\ast)$. Let $x,y\in Z_l$ and let $D$ be a Weil divisor on $X$ Cartier at $x$.
It is then also Cartier at $\eta$
and we may write $D=\sum_i\lambda_iD_i+D'$ with $D'\equiv_{\alg}0$ and $\lambda_i\in\mathbb{Z}$.
By choice of $Z_l$, $\sum_i\lambda_iD_i$ is Cartier on $Z_l$. Then $D'$ is Cartier at $x$,
hence at $y$ by the property $(\ast_0)$ of $Y_l$. This shows, as wanted, that $D$ is Cartier at $y$.
\end{proof}

We may now prove Proposition \ref{proputile}.

\begin{proof}[$\mathbf{Proof \text{ }of \text{ }Proposition\text{ }\ref{proputile}}$]~
Let $X=\bigsqcup_{k=1}^N Z_k$ be a partition as in Theorem \ref{BGSTh},
and let $\rho$ be the rank of the N\'eron-Severi group $\NS(X)$ of $X$.
Let us show that $r=2^N\rho$ works. By the pigeonhole principle, up to reordering the $U_i$, it is possible to suppose that
$U_1,\dots,U_{\rho+1}$ meet exactly the same strata $Z_k$ of the partition. 

For $1\leq i\leq \rho+1$, let $D_i$ be an effective
ample Cartier divisor on $U_i$. We still denote by $D_i$ the effective Weil divisor that is its closure in $X$.
Since $\NS(X)$ is of rank $\rho$, it is possible to find a nontrivial relation of the form 
$$\sum_{u=1}^m \alpha_{i_u}D_{i_u}\equiv_{\alg} \sum_{v=1}^n \beta_{j_v}D_{j_v},$$ where
$i_1,\dots,i_m,j_1,\dots,j_n\in\{1,\ldots, \rho+1\}$ are distinct indices and $\alpha_{i_u},\beta_{j_v}$ are positive integers.
Note that since a nonzero effective divisor cannot be algebraically equivalent to $0$, we have $m,n\geq 1$.

We set $U=(U_{i_1}\cap\dots\cap U_{i_m})\cup(U_{j_1}\cap\dots\cap U_{j_n})$,  $D=\sum_{u=1}^m \alpha_{i_u}D_{i_u}$, and
$D'=\sum_{v=1}^n \beta_{j_v}D_{j_v}$.
The Weil divisor $D_{i_u}$ is Cartier on $U_{i_u}$, hence on all $Z_k$ that meet $U_{i_u}$
by the property $(\ast)$ of the strata. By the choice of
$U_1,\dots,U_{\rho+1}$, this shows that $D_{i_u}$ is Cartier on $U_i$ for $1\leq i\leq \rho+1$.
In particular, it is Cartier on $U$. Consequently,
$D$ is Cartier on $U$. The same argument shows that $D'$ is Cartier on $U$.

Now $D$ is ample on $U_{i_1}\cap\dots\cap U_{i_m}$ and $D'$ is ample on $U_{j_1}\cap\dots\cap U_{j_n}$.
By Proposition \ref{amplessi},
$D$ is ample on both $U_{i_1}\cap\dots\cap U_{i_m}$ and $U_{j_1}\cap\dots\cap U_{j_n}$.
By \cite{Wlo} Lemma 2, it is then ample on $U$: this shows that
$U$ is quasi-projective.
\end{proof}

Let us finally explain how to remove the hypothesis that the base field is algebraically closed in Theorem \ref{princ}.

\begin{thm}\label{princK}
Let $K$ be a field and let $X$ be a separated and normal $K$-scheme of finite type. 
Then $X$ contains finitely many maximal quasi-pro\-jec\-tive open subschemes.
\end{thm}

\begin{proof}[$\mathbf{Proof}$]~
Suppose that $(U_n)_{n\in\mathbb{N}}$ are distinct maximal quasi-projective open subschemes of $X$.
Let $\bar{K}$ be an algebraic closure of $K$, let 
$\tilde{X}_{\bar{K}}$ be the normalization of $X_{\bar{K}}$ and let $\tilde{U}_{n,\bar{K}}$ be the inverse image of $U_n$ in $\tilde{X}_{\bar{K}}$.
Since the pull-back of an ample line bundle by a finite morphism is ample, the $(\tilde{U}_n)_{n\in\mathbb{N}}$ are distinct quasi-projective
open subschemes of $\tilde{X}_{\bar{K}}$. 

Applying Theorem \ref{princ} to the connected components of $\tilde{X}_{\bar{K}}$, we see that $\tilde{X}_{\bar{K}}$ contains only finitely many
maximal quasi-projective open subschemes. Hence, we may suppose that 
$\tilde{U}_{1,\bar{K}}$ and $\tilde{U}_{2,\bar{K}}$ are contained in the same maximal quasi-projective open subscheme,
so that $\tilde{U}_{1,\bar{K}}\cup\tilde{U}_{2,\bar{K}}$ is quasi-projective.

 We then use limit arguments. By \cite{EGA43}, Th\'eor\`eme 8.8.2 (ii), there exist a finite extension $L$ of $K$ and a scheme $Y$ of finite type over $L$ such that
$Y\times_L\bar{K}\simeq\tilde{X}_{\bar{K}}$. Moreover, by \cite{EGA43} Th\'eor\`eme 8.8.2 (i) and Th\'eor\`eme 8.10.5 (x), up to replacing $L$ by a finite extension,
there exists a finite morphism $Y\to X_L$ inducing the normalization
$\tilde{X}_{\bar{K}}\to X_{\bar{K}}$. Let $V$ denote the inverse image of $U_1\cup U_2$ in $Y$.
  By \cite{EGA43} Th\'eor\`eme 8.10.5 (xiv),
up to enlarging $L$ again, we may suppose that $V$ is quasi-projective.

The finite and surjective morphism $V\to U_1\cup U_2$ has normal target, so that the
hypotheses of \cite{EGA2} Corollaire 6.6.2 are satisfied, showing that $U_1\cup U_2$ is quasi-projective.
By maximality of $U_1$ and $U_2$, we must have $U_1=U_2$: this is a contradiction.
\end{proof}


\bibliographystyle{plain}
\bibliography{biblio}

\end{document}